\documentclass[12pt, reqno]{amsart}
\usepackage{amsmath, amsthm, amscd, amsfonts, latexsym, amssymb, graphicx, color}
\usepackage[bookmarksnumbered, colorlinks, plainpages]{hyperref}
\usepackage{float}

\makeatletter \oddsidemargin.9375in \evensidemargin \oddsidemargin
\newtheorem{thm}{Theorem}[section]
\newtheorem{lem}[thm]{Lemma}

\theoremstyle{definition}

\theoremstyle{remark}

\numberwithin{equation}{section}

\begin{document}

\setcounter{page}{1}


\title[Classification of 3-GNDB graphs ]{Classification of 3-GNDB graphs}
\author[Hosseini, Alaeiyan and Aliannejadi]{A. Hosseini, M. Alaeiyan and Z. Aliannejadi$^{*}$}
\thanks{{\scriptsize
\hskip -0.4 true cm MSC(2010): Primary: 05C12; Secondary: 05C40,
11Y50
$*$Corresponding author }}
\begin{abstract}
A nonempty graph $\Gamma$ is called generalized 3-distance-balanced, (3-$GDB$) whenever for every edge $ab$, $|W_{ab}|=3|W_{ba}|$ or conversely. As well as a graph $\Gamma$ is called generalized 3-nicely distance-balanced (3-$GNDB$) whenever for every edge $ab$ of $\Gamma$, there exists a positive integer $\gamma_\Gamma$, such that: $|W_{ba}|=\gamma_\Gamma$.\\
In this paper, we classify 3-$GNDB$ graphs with, $\gamma_\Gamma\in \{1,2\}$. 

\end{abstract}

\maketitle
\section{Introduction}
Throughout of this paper, let $\Gamma$ be a finite, undirected, connected graph with diameter $d$, and $V(\Gamma)$ and $E(\Gamma)$ denote the vertex and edge set of $\Gamma$, respectively. The distance $d(a,b)$ between vertices $a,b\in V(\Gamma)$ is the length of a shortest path between $a,b\in V(\Gamma)$. For an edge $ab$ of a graph $\Gamma$, let $W_{ab}$ be the set of vertices closer to $a$ than to $b$ , that is $W_{ab}=\{x\in\Gamma|d(x,a)<d(x, b)\}$. We call a graph $\Gamma$, distance-balanced $(DB)$, if $|W_{ab}|=|W_{ba}|$ for every edge $ab\in E(\Gamma)$. These graphs were studied by Handa \cite{6} who considered $DB$. For recent results on $DB$ and $EDB$ see \cite{3, 4, 5, 7, 8, 9, 11, 12}. A graph $\Gamma$ is called nicely distance-balanced, whenever there exists a positive integer $\gamma_\Gamma$, such that for two adjacent vertices $a,b$ of $\Gamma$; $|W_{ab}|=|W_{ba}|=\gamma_\Gamma$. These graphs were studied by Kutnur and Miklavi\v{c} in \cite{10}.\\
A graph $\Gamma$ is called generalized 3-distance-balanced (3-$GDB$) if for every
edge $ab\in E(\Gamma)$; $|W_{ab}|=3|W_{ba}|$ or conversely. Throughout of this
paper, we assume that $|W_{ab}|=3|W_{ba}|$. A graph $\Gamma$ is called generalized
3-nicely distance-balanced (3-$GNDB$), if for every edge $ab$ of $\Gamma$, there exists a positive integer $\gamma_\Gamma$, such that: $|W_{ba}|=\gamma_\Gamma$. For example we can show that $K_{1,3}$ and $K_{2,6}$ are $3-GNDB$. The aim of this paper is classifying 3-$GNDB$ graphs with $\gamma_\Gamma\in \{1,2\}$.

\section{Classification}

In order to express the problem, it is better to start with parameter 3. In this section, we classify 3-$GNDB$ graphs with $\gamma_\Gamma\in \{1,2\}$.\\
For every two non-negative integers $i,j$, we denote:
\begin{center}
$\hspace*{1.5cm}$ $D^i_j(a,b)=\{x\in V(\Gamma) |d(x,a)=i$ and $d(x,b)=j\}.$ $\hspace*{1.29cm}(1)$
\end{center}
We now suppose that $\Gamma$ is a 3-$GNDB$ graph with diameter $d$. Since $|W_{ab}|=3|W_{ba}|$ for every two adjacent vertices $a,b$ and by (1), we have
\begin{center}
$|\{a\}\bigcup^{d-1}_{i=1} D^i_{i+1}(a,b)|=3|\{b\}\bigcup^{d-1}_{i=1} D^{i+1}_i(a,b)|$.
\end{center}
Therefore,\\
$\hspace*{2cm}$ $\sum^{d-1}_{i=1}|D^i_{i+1}(a, b)|=3 \sum^{d-1}_{i=1}|D^{i+1}_i(a, b)|+2$.$\hspace*{2.5cm}(2)$

\begin{thm}\label{th2.1}
If $\Gamma$ be a connected $k-GNDB$ graph, then $\Gamma$ is a bipartite graph.

\begin{proof}. Inspired by the proof of Theorem 1.1 in \cite{2}, let $\Gamma$ be a $k-GNDB$ graph with diameter $d$, and the vertex set $\{v_1,v_2,...,v_{2l+1}\}$ form an odd cycle with length $2l + 1$ such that $v_iv_{i+1}\in E(\Gamma)$. Set
\begin{center}
$A_{ij}=\{v\in V(\Gamma)|d(v,v_{i+2l})=m_jk$,
\end{center}
\begin{center}
$m_jk=\{1,2,...,d\}, k=0,1,...,2l, 2\leqslant j\leqslant r\}$,
\end{center}
and
\begin{center}
$W^\Gamma_{v_i,v_{i+l}}=(\bigcup^r_{j=1} A_{ij})\bigcup\{v_i,v_{i+2l}\}$,
\end{center}
\begin{center}
$W^\Gamma_{v_{i+1},vi}=(\bigcup^r_{j=1} A_{(i+1)j})\bigcup\{v_{i+1},v_{i+2}\}$,
\end{center}
where the calculation in indexes $i$ are performed modulo $2l+1$ and some $r\in N$. Taking $|A_{ij}|=a_{ij}$ for $i=0,1,...,2l$ and $j=1,2,...,r$, by definition $k-GNDB$ graphs, there exists $e_i\in \{ \pm1\}, i=0,1,...,2l$ such that
\begin{center}
$\sum^r_{j=1} a_{0j}+2=k^{e_0}(\sum^r_{j=1} a_{1j}+2),$
\end{center}
\begin{center}
$\sum^r_{j=1} a_{1j}+2=k^{e_1}(\sum^r_{j=1} a_{2j}+2),$
\end{center}
$\hspace*{5.8cm}.$\\
$\hspace*{5.8cm}.$\\
$\hspace*{5.8cm}.$\\
\begin{center}
$\sum^r_{j=1} a_{(2l-1)j}+2=k^{e_{2l-1}}(\sum^r_{j=1} a_{2l}+2)$,
\end{center}
\begin{center}
$\sum^r_{j=1} a_{(2l)j}+2=k^{e_{2l}}(\sum^r_{j=1} a_{0j} + 2)$.
\end{center}
Now, multipling all $(2l+1)$ equations above imply that $k^{{\sum^{2i}_{i=0}} e_i}=1$,
that is, $\sum^{2i}_{i=0} ei=0$. On the other hand,
$e_i\in \{ \pm1\} \Longrightarrow 1 \leqslant |\sum^{2i}_i=0 e_i|$,
which is a contradiction and henes $\Gamma$ has no odd cycle. This completes
the proof. 
\end{proof} 
\end{thm}
\begin{thm}\label{th2.2}
If $\Gamma$ be a 3-GNDB grpah with $d=2$, then $deg(a)=3 deg(b)$ for every edge $ab$ of $\Gamma$. 
\end{thm}
\begin{proof}
It follows from (1) that for a $3-GNDB$ graph with diameter $2$, $|D^1_2(a,b)|=3|D^2_1(a,b)|+2$, for every edge $ab$ of $\Gamma$. If $|D^2_1(a,b)|=t$, then $|D^1_2(a,b)|=3t+2.$ Therefore, $deg(b)=t+1$ and $deg(a)=3t+3$. So always $deg(a)=3 deg(b)$. 
\end{proof}
\begin{lem}\label{lem2.3} Let $\Gamma$ be a $3-GNDB$ graph with diameter 2. Then $\Gamma$ is only $K_{n,3n}$. 
\end{lem}
\begin{proof}
Let $\Gamma$ be a $3-GNDB$ graph with diameter 2. We claim that $\Gamma$ is a complete bipartite graph. Otherwise, it does not have diameter 2. It follows from Theorem 2.2 that $deg(a)=3deg(b)$. Since $\Gamma$ is complete bipartite graph, $\Gamma$ must be $K_{n,3n}$. 
\end{proof}

\begin{lem}\label{lem2.4}
Let $\Gamma$ be a connected $k-GNDB$ graph with diameter $d$. Then $d \leqslant k \gamma_\Gamma$ 
\end{lem}
\begin{proof}
Pick vertices $x_0$ and $x_d$ of $\Gamma$ such that $d(x_0, x_d)=d$ and a shortest path $x_0,x_1,x_2,...,x_d$ between $x_0$ and $x_d$. We may assume without loss of generality that $|W_{x_0,x_1}|=k|W_{x_0,x_1}|$. Then $\{x_1,x_2,...,x_d\}\in W_{x_1,x_0}$. Hence $|\{x_1,x_2,...,x_d\}|\leqslant |W_{x_1,x_0}|=k|W_{x_0,x_1}|$. this shows
that $d\leqslant k\gamma_\Gamma$.

\end{proof}

We now classify $3-GNDB$ graphs $\Gamma$ with, $\gamma_\Gamma \in \{1,2\}$.\\ 
First we consider when $\gamma_\Gamma=1$. By the Lemma 2.4, $d \leqslant 3$.\\
If $d=1$, then $\Gamma$ is complete graph.\\
If $d=2$, by the Lemma 2.3, $\Gamma$ is only $K_{1,3}$.\\
If $d=3$, then we have only a path of length 2, that it would not be $3-GNDB$.\\

Now we consider the case $\gamma_\Gamma=2$.
\begin{thm}\label{th2.5}
A graph $\Gamma$ is $3-GNDB$ with $\gamma_\Gamma=2$ if and only if it is $K_{2,6}$. 
\end{thm}
\begin{proof}
For adjacent vertices $a,b$ of $\Gamma$, we say that the edge ab is consistent if $|W_{ab}|=3|W_{ba}|$. Let $d$ be the diameter of $\Gamma$. By the Lemma 2.4, $d \leqslant 6$.
If $d=1$, then $\Gamma$ is a complete graph. Therefore, $d\in \{2,3,4,5,6\}$. Pick an edge $xy\in E(\Gamma)$ and for non-negative integers $i,j$ set $D^i_j=D^i_j(x, y)$. Note that, by triangle inequality, $D^i_j=\phi$ whenever $|i-j|>1$. If $d=2$, then by Lemma 2.3, $\Gamma$ is only $K_{2,6}$.\\
Note that $|V(\Gamma)|=8$. Consider that $xy\in E(\Gamma)$ and $d\in \{3,4,5,6\}$. Therefore, $|V(\Gamma)|$ $\setminus$ $\{x,y\}=6$. Since for every $D^i_j$, in which $i,j\neq0$, then there must be at least a neighbour for either vertex $x$ or vertex $y$. Suppose that $|D^2_1|=1$ for all cases. We now consider all different cases of $|D^i_j|$, where $i,j\neq0$ for the 6 remaining vertices in $\Gamma$ and edge $xy$.
Now we show that, there is no graph for $3\leqslant d\leqslant 6$.\\ 
If $d=3$, then we split our proof into the following subcases.\\ 

\textbf{Subcase 1:} $|D^1_2|=1, |D^2_3|=4$ and $|D^2_1|=1$.\\ 
We will show that this case can not occur. Denote the vertex in $D^1_2$ by $x_1$, the vertices in $D^2_3$ by $x_2, x_3, x_4$ and $x_5$, and also the vertex in $D^2_1$ by $y_1$. The vertices $x_2$ up to $x_5$ can not be adjacent with $y_1$, because which is created an odd cycle. Therefore, the vertices $x_2$ up to $x_5$ must be adjacent to each other. In this case we have an odd cycle.\\ 

\textbf{Subcase 2:} $|D^1_2|=2, |D^2_3|=3$ and $|D^2_1|=1$.\\ 
Denote the vertices in $D^1_2$ by $x_1$ and $x_2$, the vertices in $D^2_3$ by $x_3, x_4$ and $x_5$, and also the vertex in $D^2_1$ by $y_1$. The vertices $x_3, x_4$ and $x_5$ can not be adjacent with $y_1$. The vertices $x_3, x_4$ and $x_5$ can only be adjacent with $x_2$. Since the diameter of graph is 3, the vertex $y_1$ must be adjacent with $x_1$ or $x_2$ or both. In each case, the edges $xx_1$ or $yy_1$ are not consistent. 
\begin{figure}[H] 
\includegraphics[scale=0.9]{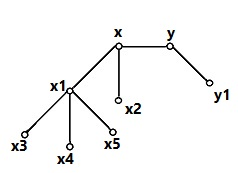}
\centering
\label{fig:1} 
\end{figure}
\textbf{Subcase 3:} $|D^1_2|=2, |D^2_3|=3$ and $|D^2_1|=1$.\\ 
Denote the vertices in $D^1_2$ by $x_1$ and $x_2$, the vertices in $D^2_3$ by $x_3, x_4$ and $x_5$, and also the vertex in $D^2_1$ by $y_1$. The vertices $x_3, x_4$ and $x_5$ can not be adjacent with $y_1$. The vertices $x_3, x_4$ and $y_1$ can not be adjacent with $x_5$. The vertices $x_3$ and $x_4$ must be adjacent with $x_2$, and vertex $x_5$ must be adjacent with $x_1$. The vertex $y_1$ can be adjacent with $x_1$ or $x_2$ or both. In each case the edge $xx_1$ is not consistent.
\begin{figure}[H] 
\includegraphics[scale=0.9]{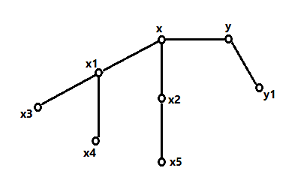}
\centering
\label{fig:1} 
\end{figure} 
\textbf{Subcase 4:} $|D^1_2|=3, |D^2_3|=2, |D^2_1|=1$.\\ 
Denote the vertices in $D^1_2$ by $x_1, x_2$ and $x_3$, the vertices in $D^2_3$ by $x_4$ and $x_5$, and also the vertex in $D^2_1$ by $y_1$. The vertex $y_1$ can not be adjacent with $x_4$ and $x_5$. The vertex $x_3$ can be adjacent with $x_4, x_5$ and $y_1$, and also The vertex $x_4$ can be adjacent with $x_2$ and $x_3$. Since the diameter of graph is 3, the vertices $y_1$ and $x_4$ must be adjacent with $x_2$ or the vertices $y_1$ and $x_5$ must be adjacent with $x_1$ or the vertex $y_1$ must be adjacent with $x_1$ and $x_2$. In each case the edges $xx_2$ and $x_2x_5$ are not consistent.
\begin{figure}[H] 
\includegraphics[scale=0.9]{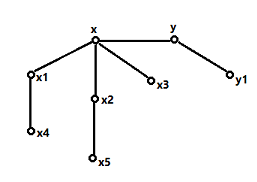}
\centering
\label{fig:1} 
\end{figure}
\textbf{Subcase 5:} $|D^1_2|=3, |D^2_3|=2, |D^2_1|=1$.\\ 
Denote the vertices in $D^1_2$ by $x_1, x_2$ and $x_3$, the vertices in $D^2_3$ by $x_4$ and $x_5$, and also the vertex in $D^2_1$ by $y_1$. The vertex $y_1$ can not be adjacent with $x_4$ and $x_5$. The vertices $x_4$ and $x_5$ can only be adjacent with $x_2$ and $x_3$. The vertices $x_2$ and $x_3$ can be adjacent with $x_4, x_5$ and $y_1$. Since the diameter of graph is 3, the vertices $y_1$ must be adjacent with $x_1$. In each case the edges $xx_1$ and $yy_1$ are not consistent.
\begin{figure}[H] 
\includegraphics[scale=0.9]{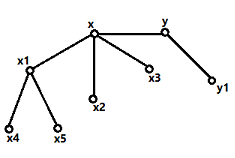}
\centering
\label{fig:1} 
\end{figure}
\textbf{Subcase 6:} $|D^1_2|=4, |D^2_3|=1, |D^2_1|=1$.\\
Denote the vertices in $D^1_2$ by $x_1, x_2, x_3$ and $x_4$, and the vertex in $D^2_3$ and $D^2_1$ by $x_5$ and $y_1$ respectively. The vertex $y_1$ can not be adjacent with $x_5$. The vertex $y_1$ can be adjacent with $x_1, x_2, x_3$ and $x_4$, and also the vertex $x_5$ can be adjacent with $x_1, x_2$ and $x_3$. In each case the edges $x_4x_5$ and $yy_1$ are not consistent.\\ 
If $d=4$ we split our proof into the following subcases.\\  

\textbf{Subcase 1:} $|D^1_2|=1, |D^2_3|=2, |D^3_4|=2, |D^2_1|=1$.\\
Denote the vertex in $D^1_2$ by $x_1$, the vertices in $D^2_3$ by $x_2$ and $x_3$, the vertices in $D^3_4$ by$x_4$ and $x_5$, and the vertex in $D^2_1$ by $y_1$. The vertex $y_1$ can not be adjacent with $x_2$ and $x_3$. The vertex $y_1$ can be adjacent with $x_1, x_4$ and $x_5$. The vertex $x_2$ can be adjacent with $x_4$ and $x_5$, and also the vertices $x_4$ and $x_5$ can only be adjacent with $x_2$ and $y_1$. In each case the edge $x_1x_3$ is not consistent.
\begin{figure}[H] 
\includegraphics[scale=0.9]{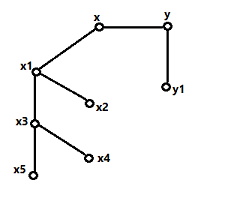}
\centering
\label{fig:1} 
\end{figure} 
\textbf{Subcase 2:} $|D^1_2|=1, |D^2_3|=2, |D^3_4|=2, |D^2_1|=1$.\\
Denote the vertex in $D^1_2$ by $x_1$, the vertices in $D^2_3$ by $x_2$ and $x_3$, the vertices in $D^3_4$ by $x_4$ and $x_5$, and the vertex in $D^2_1$ by $y_1$. The vertex $y_1$ can not be adjacent with $x_2$ and $x_3$. The vertex $y_1$ can be adjacent with $x_1, x_4$ and $x_5$. The vertex $x_4$ can be adjacent with $x_3$ and $y_1$, and also the vertex $x_5$ can be adjacent with $x_2$ and $y_1$. In each case the edge $x_2x_4$ is not consistent.
\begin{figure}[H] 
\includegraphics[scale=0.9]{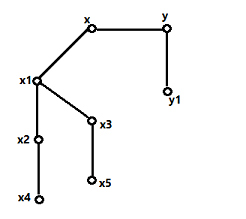}
\centering
\label{fig:1} 
\end{figure}
\textbf{Subcase 3:} $|D^1_2|=1, |D^2_3|=1, |D^3_4|=3, |D^2_1|=1.$\\ 
Denote the vertex in $D^1_2$ by $x_1$, the vertex in $D^2_3$ by $x_2$ and $x_3$, the vertices in $D^3_4$ by $x_3, x_4$ and $x_5$, and the vertex in $D^2_1$ by $y_1$. The vertex $y_1$ can not be adjacent with $x_2$. The vertices $x_3, x_4$ and $x_5$ must be adjacent with $y_1$. The vertex $y_1$ can be adjacent with $x_1, x_3,x_4$ and $x_5$. In each case the edge $x_1x_2$ is not consistent.\\ 

\textbf{Subcase 4:} $|D^1_2|=1, |D^2_3|=3, |D^3_4|=1, |D^2_1|=1.$\\  
Denote the vertex in $D^1_2$ by $x_1$, the vertices in $D^2_3$ by $x_2$, $x_3$ and $x_4$, the vertex in $D^3_4$ by $x_5$, and the vertex in $D^2_1$ by $y_1$. The vertex $y_1$ can not be adjacent with $x_2, x_3$ and $x_4$. The vertices $x_3$ and $x_4$ most be adjacent with $x_5$. The vertex $x_5$ can be adjacent with $x_3, x_4$ and $y_1$, and also the vertex $y_1$ can be adjacent with $x_1$ and $x_5$. In each case the edge $x_4x_5$ is not consistent.\\  

\textbf{Subcase 5:} $|D^1_2|=2, |D^2_3|=1, |D^3_4|=2, |D^2_1|=1$.\\
Denote the vertices in $D^1_2$ by $x_1$ and $x_2$, the vertex in $D^2_3$ by $x_3$, the vertices in $D^3_4$ by $x_4$ and $x_5$, and the vertex in $D^2_1$ by $y_1$. The vertex $y_1$ can not be adjacent with $x_3$. The vertices $x_4$ and $x_5$ must be adjacent with $y_1$. The vertex $y_1$ can be adjacent with $x_2, x_3, x_4$ and $x_5$. The vertex $x_2$ can only be adjacent with $x_3$ and $y_1$. In each case the edge $x_1x_3$ is not consistent.\\   

\textbf{Subcase 6:} $|D^1_2|=2, |D^2_3|=2, |D^3_4|=1, |D^2_1|=1$.\\  
Denote the vertices in $D^1_2$ by $x_1$ and $x_2$, the vertices in $D^2_3$ by $x_3$ and $x_4$, the vertex in $D^3_4$ by $x_5$, and the vertex in $D^1_2$ by $y_1$. The vertex $y_1$ can not be adjacent with $x_3$ and $x_4$. The vertex $x_4$ can be adjacent with $x_1$ and $x_5$. The vertex $y_1$ can be adjacent with $x_1, x_2$ and $x_5$. The vertex $x_2$ can be adjacent with $x_3$ and $y_1$. The vertex $x_3$ can only be adjacent with $x_2$. The vertex $x_5$ can be adjacent with $x_4$ and $y_1$. In each case the edge $xx_1$ is not consistent.\\ 
 
\textbf{Subcase 7:} $|D^1_2|=3, |D^2_3|=1, |D^3_4|=1, |D^2_1|=1$.\\
Denote the vertices in $D^1_2$ by $x_1, x_2$ and $x_3$, the vertex in $D^2_3$ by $x_4$, the vertex in $D^3_4$ by $x_5$ and the vertex in $D^2_1$ by $y_1$. The vertex $y_1$ can not be adjacent with $x_4$. The vertices $x_2$ and ?$?x_3$ can be adjacent with $y_1$ and $x_4$. The vertex $y_1$ can be adjacent with $x_1, x_2, x_3$ and $x_5$. The vertex $x_4$ can be adjacent with $x_2$ and $x_3$. The vertex $x_5$ must be adjacent with $y_1$. In each case the edge $xx_1$ is not consistent.\\ 
If $d=5$ we split our proof into the following subcases.\\ 

\textbf{Subcase 1:} $|D^1_2|=|D^2_3|=|D^3_4|=1, |D^4_5|=2$ and $|D^2_1|=1.$\\  
Denote the vertex in $D^1_2, D^2_3$ and $D^3_4$ by $x_1, x_2$ and $x_3$ respectively, the vertices in $D^4_5$ by $x_4$ and $x_5$, and also the vertex in $D^2_1$ by $y_1$. The vertex $y_1$ can not be adjacent with $x_2, x_4$ and $x_5$. In this case the vertices $x_4$ and $x_5$ can not be adjacent with other vertices.\\

\textbf{Subcase 2:} $|D^1_2|=|D^2_3|=1, |D^3_4|=2, |D^4_5|=1$ and $|D^2_1|=1.$\\
Denote the vertex in $D^1_2, D^2_3$ by $x_1$ and $x_2$ respectively, the vertex in $D^3_4$ by $x_3$ and $x_4$, the vertex in $D^4_5$ by $x_5$ and also the vertex in $D^2_1$ by $y_1$. The vertices $y_1$ can not be adjacent with $x_2$ and $x_5$. The vertex $x_4$ must be adjacent with $x_5$. The vertex $y_1$ can be adjacent with $x_3$ and $x_4$. In each case the edge $x_1x_2$ is not consistent.\\   

\textbf{Subcase 3:} $|D^1_2|=1, |D^2_3|=2, |D^3_4|=1, |D^4_5|=1$ and $|D^2_1|=1.$\\  
Denote the vertex in $D^1_2$ by $x_1$, the vertices in $D^2_3$ by $x_2$ and $x_3$, the vertex in $D^3_4$ by $x_4$, the vertex in $D^4_5$ by $x_5$, and also the vertex in $D^2_1$ by $y_1$. The vertex $y_1$ can not be adjacent with $x_2, x_3$ and $x_5$. In this case the vertex $x_5$ can not be adjacent with other vertices.\\ 
If $d=6$, we have: $|D^1_2|=|D^2_3|=|D^3_4|=|D^4_5|=|D^5_6|=|D^2_1|=1.$\\
Denote the vertex in $D^1_2, D^2_3, D^3_4, D^4_5, D^5_6$ and $D^2_1$ by $x_1, x_2, x_3, x_4, x_5$ and $y_1$ respectively. The vertex $y_1$ can not be adjacent with $x_2$ and $x_4$, and also the vertex $y_1$ can be adjacent with $x_1, x_3$ and $x_5$. In each case the edge $x_1x_2$ is not consistent.   
\end{proof}

\bibliographystyle{amsplain}

\bigskip
\bigskip

{\footnotesize {\bf Amir Hosseini}\; \\ {Department of Mathematics, Islamic Azad University, Nazarabad Branch,
Nazarabad, Iran. }\\
{\tt Email: hosseini.sam.52@gmail.com }\\

{\footnotesize {\bf Mehdi Alaeiyan}\; \\ {Department of Mathematics, Iran University of Science and Technology, Narmak,
Tehran 16844. Iran. }\\
{\tt Email: alaeiyan@iust.ac.ir }\\

{\footnotesize {\bf Zohreh Aliannejadi}\; \\ {Department of Mathematics, Islamic Azad University, South Tehran Branch,
Tehran, Iran. }\\
{\tt Email: z\_alian@azad.ac.ir }\\

\end{document}